\documentclass[a4paper]{amsart}

\usepackage[utf8]{inputenc}
\usepackage{amssymb}
\usepackage{amsthm}
\usepackage{amsmath}
\usepackage{mathtools}
\usepackage{enumitem}
\usepackage[maxbibnames=20]{biblatex}
\usepackage{nicefrac}
\usepackage{xcolor}

\begin{document}
	
\theoremstyle{plain}
\newtheorem{lemma}{Lemma}
\numberwithin{lemma}{section}
\newtheorem{proposition}[lemma]{Proposition}
\newtheorem{corollary}[lemma]{Corollary}
\newtheorem{theorem}[lemma]{Theorem}
	
\theoremstyle{definition}
\newtheorem{definition}[lemma]{Definition}
\newtheorem{question}[lemma]{Question}

\theoremstyle{remark}
\newtheorem{remark}[lemma]{Remark}

\newcommand{\period}{\text{.}}
\newcommand{\comma}{\text{,}}

\newcommand{\dg}{\mathsf{deg}}
\newcommand{\n}{\mathbb{N}}
\newcommand{\nn}{\n^\n}
\newcommand{\dom}{\mathsf{dom}}
\newcommand{\cont}{\mathrel{\leq^*_{\mathsf{W}}}}
\newcommand{\contEq}{\mathrel{\equiv^*_{\mathsf{W}}}}
\newcommand{\contNeq}{\mathrel{\not\equiv^*_{\mathsf{W}}}}
\newcommand{\contStrict}{\mathrel{<^*_{\mathsf{W}}}}
\newcommand{\weih}{\mathrel{\leq_{\mathsf{W}}}}
\newcommand{\weihEq}{\mathrel{\equiv_{\mathsf{W}}}}
\newcommand{\weihNeq}{\mathrel{\not\equiv_{\mathsf{W}}}}
\newcommand{\weihStrict}{\mathrel{<_{\mathsf{W}}}}
\newcommand{\wkl}{\mathsf{WKL}}
\newcommand{\lpo}{\mathsf{LPO}}
\newcommand{\llpo}{\mathsf{LLPO}}
\newcommand{\lposqrt}{\lpo^{\nicefrac{1}{2}}}
\newcommand{\id}{\mathsf{id}}
\newcommand{\one}{\mathbf{1}}
\newcommand{\zero}{\mathbf{0}}
\newcommand{\restr}[2]{#1 \mathbin{\upharpoonright}_{#2}}

\title{Weihrauch degrees without roots}
\author[P. Uftring]{Patrick Uftring}

\address{Patrick Uftring, University of W\"urzburg, Institute of Mathematics, Emil-Fischer-Stra{\ss}e~40, 97074 W\"urzburg, Germany}
\email{patrick.uftring@uni-wuerzburg.de}
\thanks{The author is funded by the Deutsche Forschungsgemeinschaft (DFG, German Research Foundation) -- Project number 460597863.}

\begin{abstract}
	We answer the following question by Arno Pauly: ``Is there a square-root operator on the Weihrauch degrees?''. In fact, we show that there are uncountably many pairwise incomparable Weihrauch degrees without any roots. We also prove that the omniscience principles of $\lpo$ and $\llpo$ do not have roots.
\end{abstract}

\maketitle

\section{Introduction}

Weihrauch reducibility captures the idea of using a mathematical problem exactly once as an oracle in order to solve some other problem in an otherwise computable manner. We give a short introduction to this topic. For more details, see~\cite{BGP}. Given two spaces $X$ and $Y$, a problem is simply given by a relation between $X$ and $Y$. We interpret any such relation $R$ as a partial multi-valued function $f:\subseteq X \rightrightarrows Y$, where $x \in X$ is in the domain of $f$ (write $\dom(f)$) if and only if there exists some $y \in Y$ with $x \mathrel{R} y$. Then, $f(x)$ is the set of all such $y$. Elements $x \in \dom(f)$ are interpreted as \emph{instances} of the problem $f$, and elements $y \in f(x)$ are interpreted as \emph{solutions} for the given instance $x$ of $f$. For example, $\wkl$ is the problem that takes any infinite binary tree as an instance and returns a path through this tree as solution. Notice that $\wkl$ is multi-valued since there may be multiple possible paths.

In the context of Weihrauch reducibility, problems are partial multi-valued functions $f:\subseteq X \rightrightarrows Y$ whose spaces $X$ and $Y$ are so-called \emph{represented spaces} (cf.~\cite{KreitzWeihrauch85}). Since we are only concerned with algebraical properties, we can restrict ourselves to the case $X := Y:= \nn$ (cf.~\cite[Lemma~11.3.8]{BGP}). Under this restriction, Weihrauch reducibility may be defined as follows:
\begin{definition}
	Given two problems $f, g:\subseteq \nn \rightrightarrows \nn$, we say that $f$ is \emph{Weihrauch reducible} to $g$ (write $f \weih g$) if and only if there are partial computable functions $h, k:\subseteq \nn \to \nn$ such that for any instance $x \in \dom(f)$, the value $k(x)$ is an instance of $g$ and any solution $y$ for $k(x)$ of $g$ results in a solution $h(\langle x, y \rangle)$ for $x$ of~$f$.
\end{definition}
Here, we write $\langle x, y \rangle$ for one of the usual (uniformly) computable ways to express two number sequences $x, y \in \nn$ as a single one. In the following, we will often say that $h$ and $k$ (in this order) \emph{realize} the Weihrauch reduction $f \weih g$.

Given two problems $f$ and $g$, we may write $f \weihEq g$ if both $f \weih g$ and $g \weih f$ hold. Taking the quotient of all problems with respect to $\weihEq$ results in the lattice structure of \emph{Weihrauch degrees} (cf.~\cite[Theorem~11.3.9]{BGP}, \cite[Theorem~3.14]{BG11}, \cite[Corollary~4.7]{Pauly10}). Finally, we write $f \weihStrict g$, if both $f \weih g$ and $f \weihNeq g$ hold.
Interesting examples of problems include the following (cf.~\cite{BishopBridges85} and \cite{BridgesRichman87} for the origins of $\lpo$ and $\llpo$, see \cite[Definition~6.1]{BG11} for the definitions that we are using):
\begin{definition}
	Let us express any natural number $n \in \n$ by a sequence $\mathbf{n} \in \nn$ where every member of $\mathbf{n}$ is equal to $n$. We define the following problems:
	\begin{itemize}
		\item The identity problem $\id: \nn \to \nn$ with
		\begin{equation*}
			\id(x) := x\period
		\end{equation*}
		\item The \emph{limited principle of omniscience} $\lpo: \nn \to \nn$ with
		\begin{equation*}
			\lpo(x) :=
			\begin{cases}
				\zero & \text{if $x$ has a zero,}\\
				\one & \text{otherwise.}
			\end{cases}
		\end{equation*}
		\item The \emph{lesser limited principle of omniscience} $\llpo:\subseteq \nn \to \nn$ where $\dom(\llpo)$ only contains number sequences that have \emph{at most one} non-zero member, with
		\begin{align*}
			\llpo(x) :={} &\{\zero \mid \text{if $x$ is zero at all \emph{even} indices}\} \cup{}\\&\{\one \mid \text{if $x$ is zero at all \emph{odd} indices}\}\period
		\end{align*}
	\end{itemize}
\end{definition}
Notice that $\llpo(\zero) = \{\zero, \one\}$ holds.
For these examples, we have the following reducibilities: $\id \weihStrict \llpo$, $\llpo \weihStrict \lpo$ (cf.~\cite[Theorem~4.2]{Weihrauch92}), and by transitivity $\id \weihStrict \lpo$. Sometimes, we want to use one problem after another. This is captured by the \emph{compositional product}:
\begin{definition}
	We define the \emph{composition} of problems $f$ and $g$ by ${f \circ g: \subseteq \nn \rightrightarrows \nn}$ with $\dom(f \circ g) := \{x \in \dom(g) \mid g(x) \subseteq \dom(f)\}$ and
	\begin{equation*}
		(f \circ g)(x) := \{z \in \nn \mid \text{there exists $y \in g(x)$ with $z \in f(y)$}\}\period
	\end{equation*}
	For arbitrary problems $f$ and $g$, we write $f * g$ for the \emph{compositional product}, i.e., the degree satisfying
	\begin{equation*}
		f * g \weihEq \max{}_{\weih}\{f' \circ g' \mid f' \weih f \text{ and } g' \weih g\}\period
	\end{equation*}
\end{definition}
Intuitively, $f * g$ is the problem where we first apply $g$ and then $f$ by transforming the solution given by $g$ into an instance of $f$.
The compositional product was first defined in \cite[Section~4]{BGM12}. Since maxima (or suprema) of sets of problems do not exist in general, the existence of compositional products for any given pair of problems had to be proven, which was done in \cite[Corollary~3.7]{BP18}. Moreover, compositions and compositional products enjoy properties like associativity (cf.~\cite[Proposition~2.4.1]{Bra98} and \cite[Proposition~4.2]{BP18}), and $\id$ is a neutral element with respect to $*$ (cf.~\cite[Observation~4.3]{BP18}).
With (compositional) products defined, we can ask about roots:
\begin{definition}
	Given a problem $f$ and a number $n \in \n$, let us write $f^{[n]}$ for the compositional product of $n$-many copies of $f$ (we set $f^{[0]} := \id$). Given a number $n \geq 2$, we call a problem $r$ an $n$-th root of $f$ if and only if $f \weihEq r^{[n]}$ holds.
\end{definition}
Now, we have all the necessary ingredients in order to talk about the following open question by Arno Pauly (cf.~\cite{Pauly20}):
\begin{question}
	Is there a square-root operator on the Weihrauch degrees?
\end{question}
We show that such an operator cannot exist by proving the following theorem:
\begin{theorem}\label{thm:uncount_no_root}
	There are uncountably many pairwise incomparable Weihrauch degrees that do not have an $n$-th root for any $n \geq 2$.
\end{theorem}
Moreover, we can also show that this result does not only hold for artificial Weihrauch degrees:
\begin{theorem}\label{thm:lpo_no_root}
	The problems $\lpo$ and $\llpo$ do not have an $n$-th root for any number $n \geq 2$.
\end{theorem}
Finally, there are problems that only have some roots:
\begin{theorem}\label{thm:some_roots}
	For any $n \geq 2$, there is a problem that has an $n$-th root but no $(n+1)$-th root.
\end{theorem}

\subsection*{Acknowledgements}
I would like to thank Arno Pauly and Giovanni Sold\`{a} for our correspondence. Also, I would like to thank Nicholas Pischke without whom some of these results would still be in some drawer.

\section{Using Turing degrees}

In this section, we prove Theorem \ref{thm:uncount_no_root} using the following family of problems:
\begin{definition}
	Let $a$ be a non-zero Turing degree. We define $w_a: \nn \rightrightarrows \nn$ with
	\begin{equation*}
		w_a(x) :=
		\begin{cases}
			\{y \in \nn \mid \text{$y$ is non-computable}\} & \text{ if $x$ is computable,}\\
			\{y \in \nn \mid \text{$y$ has Turing degree $a$}\} & \text{ if $x$ is non-computable.}
		\end{cases}
	\end{equation*}
\end{definition}
Before we state and prove all required lemmas, let us give a short sketch of the proof idea: Consider $w_a$ for some non-zero Turing degree $a$ and assume that $r$ is one of its roots. First, we notice that $w_a$ has a computable instance such that all of its solutions are non-computable. (Obviously, this holds for any instance of $w_a$.) We show that this property must also hold for $r$. Then, we consider $w_a \circ r$ and easily conclude that this problem has a computable instance such that all of its solutions have a degree equal to $a$. By $w_a \circ r \weih w_a * r \weih r * w_a$, we conclude that $r * w_a$ must have a computable instance such that all of its solutions have a degree greater or equal to $a$. From this, we can extract some computable function $e$ such that all solutions of $(r \circ e)(x)$ have a degree greater or equal to $a$ for any non-computable $x$. We conclude that $r \circ e \circ r$ has a computable instance such that all of its solutions have a degree greater or equal to $a$. Finally, with $r \circ e \circ r \weih r * \id * r \weih w_a$, we conclude that the same must hold for $w_a$. This clearly contradicts the definition of~$w_a$.

When working with compositional products, we have to be extra careful because of the linear nature of Weihrauch reducibility (cf.~\cite[Section~11.9.1]{BGP}). For example, there are problems $f$ and $g$ such that we cannot reduce $f$ to $f * g$ or $g * f$: Take $f := \id$ and let $g$ be the problem with empty domain.

In the following, we are looking at roots $r$ of problems with a domain that contains computable instances. Thus, $r$ must also contain some computable element in its domain. We conclude $\id \weih r$. Now, given any problem $f$, we can always derive $f \weih f * \id \weih f * r$ and $f \weih \id * f \weih r * f$. In particular, we have $r * r \weih r^{[n]}$ for $n \geq 2$. We will apply this observation implicitly.

The following first step will be used to show that our family contains uncountably many incomparable problems.
\begin{lemma}\label{lem:incomparable}
	Let $a$ and $b$ be two incomparable Turing degrees. Then, the problems $w_a$ and $w_b$ are also incomparable.
\end{lemma}

\begin{proof}
	We prove that $w_a \weih w_b$ implies that $a$ is Turing reducible to $b$ for two non-zero Turing degrees $a$ and $b$.
	Assume that $h, k:\subseteq \nn \to \nn$ realize the Weihrauch reduction $w_a \weih w_b$. Let $x \in \nn$ have degree $b$. Then, $x$ itself is a solution for the instance $k(x)$ of $w_b$. Thus, $h(\langle x, x \rangle)$ is a solution for the instance $x$ of $w_a$. Since this solution is Turing reducible to $b$, our claim follows from the fact that any solution for the instance $x$ of $w_a$ has degree $a$.
\end{proof}
Next, we show that having a computable instance that only has non-computable solutions (or only solutions above a certain degree) is transferred to higher Weihrauch degrees.
\begin{lemma}\label{lem:cmp_inst_deg_sol}
	Let $a$ be some Turing degree and let $f$ and $g$ be problems with $f \weih g$. If $f$ has a computable instance such that the degree $b$ of any solution satisfies $b \geq a$ (or $b > a$), then $g$ also has a computable instance with the same property.
\end{lemma}

\begin{proof}
	Let $x \in \nn$ be a computable instance of $f$ such that all solutions in $f(x)$ have a degree $b$ with $b \geq a$ (or $b > a$). Moreover, let $h, k:\subseteq \nn \to \nn$ realize the Weihrauch reduction $f \weih g$. Given an arbitrary solution $y$ of $g(k(x))$, we know that the degree $b$ of the solution $h(\langle x, y \rangle)$ must satisfy $b \geq a$ (or $b > a$). Thus, $b$ is Turing reducible to the degree $c$ of $y$. Finally, we conclude that the degree $c$ of any solution $y$ for the computable instance $k(x)$ of $g$ must satisfy $c \geq a$ (or $c > a$).
\end{proof}
Now, we see that having some computable instance that only has non-computable solutions is in a certain sense \emph{atomic}, i.e., roots inherit this property.
\begin{lemma}\label{lem:r_cmp_inst_ncmp_sol}
	Let $a$ be some non-zero Turing degree and let $n \in \n$ with $n \geq 2$ be such that $r$ is an $n$-th root of $w_a$. Then, $r$ has a computable instance whose solutions are all non-computable.
\end{lemma}

\begin{proof}
	For contradiction, assume that all computable instances of $r$ have a computable solution.
	Using this assumption and Lemma \ref{lem:cmp_inst_deg_sol} (applied to the degree $a := 0$), we show that all computable instances of $r^{[k]}$ have a computable solution for $k \in \n$ with $k \geq 1$, by induction: For $k := 1$, this holds by assumption. Now, assume that our claim has already been shown for $k \geq 1$. Let $f \weih r$ and $g \weih r^{[k]}$ be such that $f \circ g$ has Weihrauch degree $r^{[k+1]}$. Let $x \in \nn$ be a computable instance of $f \circ g$. By induction hypothesis, $g(x)$ has a computable solution $y \in \nn$. By assumption on $r$, $f(y)$ has a computable solution $z \in \nn$. We conclude that $(f \circ g)(x)$ has a computable solution $z$. Applying Lemma \ref{lem:cmp_inst_deg_sol} to the reduction $r^{[k+1]} \weih f \circ g$ yields our claim that all computable instances of $r^{[k+1]}$ have a computable solution. However, this is clearly false for $r^{[n]} \weihEq w_a$.
\end{proof}
\begin{lemma}\label{lem:omit_computable}
	Let $f$ be a problem and $h,k:\subseteq \nn \to \nn$ two computable functions. Then, we have the reduction $h \circ f \circ k \weih f$.
\end{lemma}
\begin{proof}
	The Weihrauch reduction is realized directly by the computable functions $h$ and $k$: Let $x \in \dom(h \circ f \circ k)$. By definition of the domain of compositions, $k(x)$ is defined and an instance of $f$. Now, let $y$ be an arbitrary solution for the instance $k(x)$ of $f$. Again, by definition of the domain of compositions, $y$ must be in the domain of $h$. Finally, $h(y) \in (h \circ f \circ k)(x)$ clearly is a solution for the instance $x$ of $h \circ f \circ k$.
\end{proof}
Finally, we combine everything in order to prove Theorem \ref{thm:uncount_no_root}.
\begin{proof}[Proof of Theorem \ref{thm:uncount_no_root}]
	We show that for any Turing degree $a$, the problem $w_a$ has no $n$-th root for $n \in \n$ with $n \geq 2$. Then, our claim follows from Lemma \ref{lem:incomparable} and the well-known fact that there are uncountably many pairwise incomparable Turing degrees (cf.~\cite{Shoenfield}).
	
	Assume that $r$ is an $n$-th root of $w_a$ for some non-zero Turing degree $a$. Let $f \weih r$ and $g \weih w_a$ be such that $f \circ g$ has Weihrauch degree $r * w_a$. We have the chain of reductions $w_a \circ r \weih w_a * r \weih r * w_a \weih f \circ g$. From Lemma \ref{lem:r_cmp_inst_ncmp_sol}, we know that $r$ has a computable instance such that all of its solutions are incomputable. Thus, $w_a \circ r$ has a computable instance such that all of its solutions have a degree that is greater or equal to $a$. Via Lemma \ref{lem:cmp_inst_deg_sol}, this property is transferred to $f \circ g$. We define $c \in \nn$ to be such a computable instance of $f \circ g$.
	
	Let $h, k: \subseteq \nn \to \nn$ realize the Weihrauch reduction $g \weih w_a$. Since $k(c)$ is computable, every non-computable $x \in \nn$ (which is a solution of $w_a(k(c))$) can be converted into a solution $h(\langle c, x \rangle)$ of $g(c)$. Let us write this process in form of a computable function: Let $e: \subseteq \nn \to \nn$ be the computable map defined by $e(x) := h(\langle c, x \rangle)$ for any non-computable $x \in \nn$.\footnote{Since $e$ is computable and defined for all non-computable sequences, its domain can actually be extended to $\nn$.} We have that $e(x)$ is a solution of $g(c)$ for any non-computable $x$.
	
	Recall that $r$ has a computable instance such that all its solutions are non-computable.
	Thus, $f \circ e \circ r$ has the same instance (since the range of $e$ is a subset of $g(c) \subseteq \dom(f)$) and all of its solutions have a degree that is greater or equal to $a$. Recall $f \weih r$. Also, we have $e \circ r \weih r$ simply by Lemma \ref{lem:omit_computable}. By definition of $r * r$, this entails $f \circ (e \circ r) \weih r * r \weih r^{[n]} \weih w_a$. By Lemma~\ref{lem:cmp_inst_deg_sol}, $w_a$ must have a computable instance that only has solutions of a degree greater or equal to $a$. However, inspecting the definition of $w_a$ reveals that this is not the case.
\end{proof}

\section{Using continuity}

In this section, we prove Theorem \ref{thm:lpo_no_root}.
While the arguments of the previous section relied on Turing degrees, our next proofs use continuity. Our arguments will make use of so-called \emph{continuous} Weihrauch reducibility. This is defined like regular Weihrauch reducibility but now the functions $h, k$ that realize the reducibility only have to be continuous and not necessarily computable. For problems $f$ and $g$, we write $f \cont g$ if and only if $f$ continuously Weihrauch reduces to $g$. Similarly, we write $f \contEq g$ if and only if both $f \cont g$ and $g \cont f$ hold. Taking the quotient of the structure of problems with respect to $\contEq$ leads to the \emph{continuous Weihrauch degrees}. Finally, we write $f \contStrict g$ if and only if both $f \cont g$ and $f \contNeq g$ hold.

Let us, again, give a short sketch of the proof idea for $\lpo$ (for $\llpo$ it is quite similar): First, we introduce a notion of weak continuity that is preserved by composing problems and that is transferred to lower Weihrauch problems. Then, we show that $\lpo$ is weakly \emph{discontinuous}. Therefore, any root $r$ of $\lpo$ must also be weakly discontinuous. We prove that $\lpo$ \emph{continuously} Weihrauch reduces to any weakly discontinuous problem, in particular, to $r$. Thus, we conclude ${\lpo * \lpo \cont \lpo}$, a statement whose falsity is well-known.

\begin{definition}
	Given a number sequence $x \in \nn$ and a number $n \in \n$, let us write $\restr{x}{n}$ for the initial segment of $x$ of length $n$.
	
	Let $f:\subseteq \nn \rightrightarrows \nn$ be a partial multivalued function and let $k \in \n$ be positive. We say that $f$ is \emph{$k$-weakly continuous} if and only if for any element $x \in \dom(f)$ and sequence $(y_n)_{n \in \n} \subseteq \dom(f)$ with $\lim_{n \to \infty} y_n = x$, there exists a solution $u \in f(x)$ such that for any $l < k$ and $m \in \n$, we can find $n \geq m$ together with a solution $v \in f(y_{n \cdot k + l})$ with $\restr{u}{m} = \restr{v}{m}$.
	We say that $f$ is \emph{$k$-weakly discontinuous} if it is not $k$-weakly continuous.
\end{definition}

\begin{lemma}\label{lem:discontinuous}\mbox{}
	\begin{enumerate}[label=(\roman*)]
		\item $\lpo$ is $1$-weakly discontinuous.
		\item $\llpo$ is $2$-weakly discontinuous.
	\end{enumerate}
\end{lemma}

\begin{proof}
	For $\lpo$, choose $x := \one$ and let $y_n$ be the sequence with $(y_n)_n = 0$ and $(y_n)_m = 1$ for all $m \neq n$, for all $n \in \n$. Clearly, $(y_n)_{n \in \n}$ converges to $x$. Now, for any solution $u \in f(x)$, i.e.~$u = \one$, there exists $m := 1$ such that for all $n \geq m = 1$ and all solutions $v \in f(y_n)$, i.e.~$v = \zero$, we have $\restr{u}{m} = \restr{\one}{m} \neq \restr{\zero}{m} = \restr{v}{m}$.
	
	For $\llpo$, choose $x := \zero$ and let $y_n$ be the sequence with $(y_n)_n = 1$ and ${(y_n)_m = 0}$ for all $m \neq n$, for all $n \in \n$. Similar to before, $(y_n)_{n \in \n}$ converges to $x$. Notice that for every \emph{even} $n$, there is a $0$ at every \emph{odd} position in $y_n$, and for every \emph{odd} $n$, there is a $0$ at every \emph{even} position in $y_n$. This entails $f(y_n) = \{\one\}$ for even $n$ and $f(y_n) = \{\zero\}$ for odd $n$.
	Now, for any solution $u \in f(x)$, i.e.~$u = \zero$ (or $u = \one)$, there exist $l := 0$ (or $l := 1$) and $m := 1$ such that for all $n \geq m = 1$ and all solutions $v \in f(y_{n \cdot k + l})$, i.e.~$v = \one$ (or $v = \zero)$, we have $\restr{u}{m} = \restr{\zero}{m} \neq \restr{\one}{m} = \restr{v}{m}$ (or $\restr{u}{m} = \restr{\one}{m} \neq \restr{\zero}{m} = \restr{v}{m}$).
\end{proof}

\begin{lemma}\label{lem:closure}
	Let $f, g:\subseteq \nn \rightrightarrows \nn$ be two partial multivalued functions and let $k \in \n$ be a positive number.
	\begin{enumerate}[label=(\roman*)]
		\item If $f \cont g$ holds and $g$ is $k$-weakly continuous, then so is $f$.
		\item If both $f$ and $g$ are $k$-weakly continuous, then so is $f \circ g$.
	\end{enumerate}
\end{lemma}
In the presence of this lemma, we see that being $k$-weakly continuous actually is a property that transfers to any other problem of the same equivalence class, i.e., it is a property of the whole (continuous) Weihrauch degree.
\begin{proof}
	For (i), let $h, k:\subseteq \nn \to \nn$ realize the continuous Weihrauch reduction $f \cont g$. In order to avoid naming collisions, let us say that $g$ is $i$-weakly continuous for $i \geq 1$. Given $x \in \dom(f)$ and $(y_n)_{n \in \n} \subseteq \dom(f)$ with $\lim_{n \to \infty} y_n = x$, we use the continuity of $k$, which yields $k(x) \in \dom(g)$ and $(k(y_n))_{n \in \n} \subseteq \dom(g)$ with $\lim_{n \to \infty} k(y_n) = k(x)$. Now, we apply the assumption that $g$ is $i$-weakly continuous. This provides a solution $u \in g(k(x))$ that satisfies the requirements of the continuity for $g$. Using the Weihrauch reducibility, we know that $u' := h(\langle x, u \rangle)$ is a solution of $f(x)$. Let $l < k$ and $m' \in \n$ be arbitrary. Using the continuity of $h$, we know that there must be some $m_1 \in \n$ such that $h$ only uses the first $m_1$-many members of $x$ and $u$ in order to compute $u'$. Moreover, since we have $\lim_{n \to \infty} y_n = x$, we can find $m_2 \in \n$ such that $\restr{y_n}{m_1} = \restr{x}{m_1}$ holds for all $n \geq m_2$. We take the maximum ${m := \max(m_1, m_2)}$. Using the $i$-weakly continuity of $g$, we can find $n \geq m$ together with $v \in g(k(y_{n \cdot i + l}))$ satisfying $\restr{u}{n} = \restr{v}{n}$. Using the Weihrauch reducibility, we find that $v' := h(\langle y_{n \cdot i + l}, v \rangle)$ is a solution for the instance $y_{n \cdot i + l}$ of $f$. We prove $\restr{u'}{m'} = \restr{v'}{m'}$: First, $\restr{y_{n \cdot i + l}}{m_1} = \restr{x}{m_1}$ holds because of $n \cdot i + l \geq n \geq m \geq m_2$. Second, $\restr{u}{m_1} = \restr{v}{m_1}$ holds because of $\restr{u}{n} = \restr{v}{n}$ and $n \geq m_1$. By definition of $m_1$, this entails $\restr{u'}{m'} = \restr{v'}{m'}$.
	
	For (ii), let $x \in \dom(f \circ g)$ and $(y_n)_{n \in \n} \subseteq \dom(f \circ g)$ with $\lim_{n \to \infty} y_n = x$. We use that fact that $g$ is $k$-weakly continuous: Let $u \in g(x)$ be such that for any $l < k$ and $m \in \n$, we can find $n \geq m$ with a solution $v \in g(y_{n \cdot k + l})$ satisfying $\restr{u}{n} = \restr{v}{n}$. From this, we can define a family of indices $(n_{(l, m)})_{(l, m) \in \n \times \n}$ with $n_{(l, m)} \geq m$ for all $m \in \n$ and $l < k$ together with a sequence of elements $(v_i)_{i \in \n}$ with ${v_{m \cdot k + l} \in g(y_{n_{l, m} \cdot k + l})}$ such that $\restr{u}{m} = \restr{v_{m \cdot k + l}}{m}$ holds for all $m \in \n$ and $l < k$. Clearly, $(v_i)_{i \in \n}$ converges to $u$. Thus, we can use the fact that $f$ is $k$-weakly continuous: This yields ${u' \in f(u) \subseteq (f \circ g)(x)}$ such that for any $l < k$ and $m \in \n$, we can find $i \geq m$ together with a solution $v' \in f(v_{i \cdot k + l}) \subseteq (f \circ g)(y_{n_{(l, i)} \cdot k + l})$ with $\restr{u'}{m} = \restr{v'}{m}$. We conclude that for any $l < k$ and $m \in \n$, we can find ${j := n_{(l, i)} \geq i \geq m}$ together with a solution $v' \in (f \circ g)(y_{j \cdot k + l})$ with $\restr{u'}{m} = \restr{v'}{m}$.
\end{proof}

\begin{lemma}\label{lem:program_lpo}
	Let $f$ be a problem that is $1$-weakly discontinuous. Then, we have $\lpo \cont f$.
\end{lemma}

\begin{proof}
	If $f$ is $1$-weakly discontinuous, then there exist an element $x \in \dom(f)$ and a sequence ${(y_n)_{n \in \n} \subseteq \dom(f)}$ with $\lim_{n \to \infty} y_n = x$ such that for any $u \in f(x)$ there exists some $m \in \n$ such that for all $n \geq m$ any solution $v \in f(y_n)$ satisfies $\restr{u}{m} \neq \restr{v}{m}$.
	
	For the construction of the continuous function $k: \subseteq \nn \to \nn$ that produces instances of $f$, we set
	\begin{equation*}
		k(z) :=
		\begin{cases}
			x & \text{ if $z$ has no zero,}\\
			y_n & \text{ if the first zero in $z$ is at index $n \in \n$.}
		\end{cases}
	\end{equation*}
	Let us quickly check that $k$ is continuous: Let $z \in \nn$ and $n \in \n$ be arbitrary. If $z$ does not contain a zero, then let $m \in \n$ be an index such that $\restr{x}{n} = \restr{y_{m'}}{n}$ holds for all $m' \geq m$. Such an index $m$ exists since $(y_n)_{n \in \n}$ converges to $x$. Now, any $z' \in \nn$ with $\restr{z}{m} = \restr{z'}{m}$ does not have a zero at an index below $n$. Thus, $k(z')$ either maps to $x$ or $y_{m'}$ for some $m' \geq m$. We conclude $\restr{k(z)}{n} = \restr{k(z')}{n}$. Otherwise, if $z$ does contain a zero, let $m \in \n$ be the first such index. Now, any $z'$ with $\restr{z}{(m+1)} = \restr{z'}{(m + 1)}$ also has its first zero at index $m$. We conclude $k(z) = k(z')$.
	
	For the construction of the continuous function $h: \subseteq \nn \to \nn$ that produces solutions of $\lpo$, consider all $u \in f(x)$ such that there exists some $m \in \n$ such that for all $n \geq m$ any solution $v \in f(y_n)$ satisfies $\restr{u}{m} \neq \restr{v}{m}$. For any such $u \in f(x)$ and $m \in \n$, let us collect the finite number sequence $\restr{u}{m}$. Since there can only be at most countably infinitely many such number sequences, we can collect them in a list $L$ that we may code in form of an infinite number sequence. Let us assume that this number sequence is available to us in form of an oracle. We will provide the definition of $h$ in form of a computation that uses this oracle. Consequently, $h$ will be continuous.
	
	Let $z \in \nn$ be an instance of $\lpo$ and let $w \in \nn$ be the solution for $k(z)$ of $f$. The computation works in stages: At stage $s \in \n$, take the next finite number sequence $t$ of length $m$ from $L$. First, we check if there is a zero in $z$ with index below $m$ or $s+1$. If this is the case, we terminate and return $\zero$. Otherwise, we check if $t$ is an initial segment of $w$. If this is the case, we terminate and return $\one$.
	
	Let us verify that this procedure must eventually terminate for any input $z \in \nn$: If $z$ does contain a zero at position $n \in \n$, then this will be found out by stage $n + 1$ if the program did not already terminate at an earlier point. If $z$ does not contain a zero, then the instance of $f$ produced by $k(z)$ is $x$. Thus, $w$ is a solution for $f(x)$. Therefore, $L$ must contain an initial segment of $w$. At some point during the execution, this initial segment will be considered and will make our program terminate with output~$\one$.
	
	Finally, we show that our program always gives the right answer: If it returns $\zero$, then, by definition, this can only be the case if we have actually found a zero in $z$. If it returns $\one$, then we have found a finite number sequence $t$ of length $m$ in $L$ such that $t$ is an initial segment of $w$. Moreover, we know that $z$ does not contain a zero at an index below $m$ since we explicitly check for that. This will be important in a moment. By definition of $t$ and $m$ in $L$, we know that for all $n \geq m$ any solution $v \in f(y_n)$ satisfies $t \neq \restr{v}{m}$. Thus, $w$ cannot be a solution for the instance $y_n$ of $f$ for $n \geq m$. Therefore, $k(z)$ must be different from $y_n$ for $n \geq m$. By definition of $k$, this entails that any zero in $z$ must have an index below $m$. However, we explicitly ensured that this is not the case. We conclude that $z$ does not contain any zeros.
\end{proof}

\begin{lemma}\label{lem:program_llpo}
	Let $f$ be a problem that is $2$-weakly discontinuous. Then, we have $\llpo \cont f$.
\end{lemma}

\begin{proof}
	If $f$ is $2$-weakly discontinuous, then there exist an element $x \in \dom(f)$ and a sequence $(y_n)_{n \in \n} \subseteq \dom(f)$ with $\lim_{n \to \infty} y_n = x$ such that for any $u \in f(x)$ there exist $l < 2$ and $m \in \n$ such that for all $n \geq m$ any solution $v \in f(y_{2n + l})$ satisfies $\restr{u}{m} \neq \restr{v}{m}$. For the construction of the continuous function $k: \subseteq \nn \to \nn$ that produces instances of $f$, we choose
	\begin{equation*}
		k(z) :=
		\begin{cases}
			x & \text{ if $z$ only consists of zeros,}\\
			y_n & \text{ if the first non-zero value in $z$ is at index $n \in \n$.}
		\end{cases}
	\end{equation*}
	The proof that $k$ is continuous works similar to the matching step in the proof of Lemma~\ref{lem:program_lpo}.
	
	For the construction of the continuous function $h: \subseteq \nn \to \nn$ that produces solutions of $\llpo$, consider all $u \in f(x)$ such that there exist $l < 2$ and $m \in \n$ such that for all $n \geq m$ any solution $v \in f(y_{n \cdot 2 + l})$ satisfies $\restr{u}{m} \neq \restr{v}{m}$. Similar to before, for any such $u \in f(x)$ and $m \in \n$, we collect the finite number sequence $\restr{u}{m}$. Additionally, we remember the value of $l$, i.e., $L$ consists of pairs $(\restr{u}{m}, l)$ for any such $u$, $l$, and $m$.
	
	Let $z \in \nn$ be the instance of $\llpo$ and let $w \in \nn$ be the solution for $k(z)$ of $f$. Similar to before, we work in stages: At stage $s \in \n$, take the next pair $(t, l)$ consisting of a finite number sequence $t$ of length $m$ and a value $l < 2$ from $L$. First, we check if there is a non-zero value in $z$ with index below $2m$ or $s + 1$. If this is the case, we terminate and return $\zero$ (or $\one)$ if this index is odd (or even). Otherwise, we check if $t$ is an initial segment of $w$. If this is the case, we terminate and return $\zero$ (or $\one$) if $l$ is equal to zero (or one).
	
	We verify that this procedure terminates for any input $z \in \nn$. If $z$ does contain a non-zero value at index $n \in \n$, then our program terminates at stage $n+1$ if it has not already terminated. Otherwise, if $z$ only consists of zeros, then, the instance for $f$ produced by $k$ is equal to $x$. Thus, $w$ is a solution for $f(x)$. By definition of $L$, there must be some stage at which we choose a pair $(t, l)$ from $L$ such that $t$ is an initial segment of $w$. This will also make our program terminate.
	
	Finally, we show that our program always gives a correct answer: If it terminates because it has found a non-zero value in $z$ at index $n \in \n$, then, by definition of the domain of $\llpo$, we know that this is the only index at which $z$ can have a non-zero value. If $n$ is odd, then all even indices of $z$ must be equal to zero. Thus, $\zero$, which is the output of our program in this case, is the only correct answer. Similarly, if $n$ is even, our program returns the only correct answer $\one$.
	
	If it terminates because it has found a pair $(t, l)$ in $L$ such that $t$ of length $m$ is an initial segment of $w$, then we also know that every number in $z$ with index below $2m$ is equal to zero. If our program returns $\zero$, then we can infer $l = 0$. Thus, by definition of $t$, $l$, and $m$ in $L$, we know that for all $n \geq m$ any solution $v \in f(y_{2n + 0})$ satisfies $t \neq \restr{v}{m}$. Therefore, $w = y_i$ can only hold for indices $i \in \n$ that are odd or satisfy $i < 2m$. From the definition of $k$, we conclude that if $z$ has a non-zero member, then its index must be odd or lie below $2m$. Since we have already ensured that the latter case does not hold, we know that $z$ either has no non-zero member or only at an odd position. In both cases, $\zero$ is a correct solution. The argument for output $\one$, which entails $l = 1$, works analogously.
\end{proof}

\begin{lemma}\label{lem:omit_continuous}
	Let $f$ be a problem and $h, k:\subseteq \nn \to \nn$ two continuous functions. Then, we have the reduction $h \circ f \circ k \cont f$.
\end{lemma}
\begin{proof}
	The proof works like that of Lemma \ref{lem:omit_computable} if we replace the term ``computable'' with ``continuous''.
\end{proof}

\begin{lemma}\label{lem:combine_cont}
	Let $f_1, f_2, g_1, g_2:\subseteq \nn \rightrightarrows \nn$ be problems satisfying $f_1 \cont g_1$ and $f_2 \cont g_2$. Then, we have $f_1 * f_2 \cont g_1 * g_2$.
\end{lemma}
In order to prove this lemma (and also to state some later results), we introduce some common constructions of problems (cf.~\cite[Definition~2.3.1]{Bra98}):
\begin{definition}
	Let $f$ and $g$ be problems. We define
	\begin{enumerate}[label=(\roman*)]
		\item $\langle f, g \rangle:\subseteq \nn \rightrightarrows \nn$ with $\dom(\langle f, g \rangle) := \dom(f) \cap \dom(g)$ and
		\begin{equation*}
			\langle f, g \rangle(x) := \{\langle u, v \rangle \mid u \in f(x) \text{ and } v \in g(x)\}\period
		\end{equation*}
		\item $f \times g:\subseteq \nn \rightrightarrows \nn$ with $\dom(f \times g) := \{\langle x, y \rangle \mid x \in \dom(f) \text{ and } y \in \dom(g)\}$ and
		\begin{equation*}
			(f \times g)(\langle x, y \rangle) := \{\langle u, v \rangle \mid u \in f(x) \text{ and } v \in g(y)\}\period
		\end{equation*}
	\end{enumerate}
\end{definition}
Given problems $f$, $g$, $h$, and $k$, we can easily calculate
\begin{equation*}
	\langle f \circ g, h \circ k \rangle = (f \times h) \circ \langle g, k \rangle
\end{equation*}
such that even the domains are the same on both sides. Moreover, if a (continuous) Weihrauch reduction $f \weih g$ (or $f \cont g$) between problems $f$ and $g$ is realized by functions $h, k:\subseteq \nn \to \nn$, then we have
\begin{equation*}
	f(x) \supseteq (h \circ \langle \id, g \circ k \rangle)(x)
\end{equation*}
for all $x \in \dom(x)$.
\begin{proof}[Proof of Lemma \ref{lem:combine_cont}]
	W.l.o.g, assume that $f_1 \circ f_2 \weihEq f_1 * f_2$ holds. Otherwise, let $f'_1$ and $f'_2$ be problems with $f'_1 \weih f_1$, $f'_2 \weih f_2$ and $f'_1 \circ f'_2 \weihEq f_1 * f_2$ and continue with those. At the end, our claim follows from $f_1 * f_2 \weihEq f'_1 \circ f'_2 \cont g_1 * g_2$.
	
	Let $h_1, k_1$ be the continuous functions that realize the reduction $f_1 \cont g_1$ and let $h_2, k_2$ be such functions that realize $f_2 \cont g_2$. Clearly, we have
	\begin{align*}
		f_1 \circ f_2 &\weih h_1 \circ \langle \id, g_1 \circ k_1 \rangle \circ h_2 \circ \langle \id, g_2 \circ k_2 \rangle\\
		\intertext{by simply using the identity functions as realizers for the Weihrauch reduction. Now, we expand $\langle \id, g_i \circ k_i \rangle$ for $i \in \{1, 2\}$.}
		&= h_1 \circ (\id \times g_1) \circ \langle \id, k_1 \rangle \circ h_2 \circ (\id \times g_2) \circ \langle \id, k_2 \rangle\\
		\intertext{We define a continuous function $p:\subseteq \nn \to \nn$ with $p = \langle \id, k_1 \rangle \circ h_2$.}
		&= h_1 \circ (\id \times g_1) \circ p \circ (\id \times g_2) \circ \langle \id, k_2 \rangle\\
		\intertext{Using Lemma \ref{lem:omit_continuous}, we can omit both $h_1$ and $\langle \id, k_2 \rangle$.}
		&\cont (\id \times g_1) \circ p \circ (\id \times g_2)\\
		\intertext{Let $x \in \nn$ be an oracle that computes $p$, i.e., let $p':\subseteq \nn \to \nn$ be computable such that $p = p' \circ \langle c_x, \id \rangle$ holds, where $c_x$ is the problem that maps everything to~$x$.}
		&= (\id \times g_1) \circ p' \circ \langle c_x, \id \rangle \circ (\id \times g_2)\\
		\intertext{Some calculation reveals the equality $\langle c_x, \id \rangle \circ (\id \times g_2) = (\id \times (\id \times g_2)) \circ \langle c_x, \id \rangle$.}
		&= (\id \times g_1) \circ p' \circ (\id \times (\id \times g_2)) \circ \langle c_x, \id \rangle\\
		\intertext{We apply Lemma \ref{lem:omit_continuous} for a second time.}
		&\cont (\id \times g_1) \circ p' \circ (\id \times (\id \times g_2))\\
		&\weih g_1 * \id * g_2 \weihEq g_1 * g_2\period\qedhere
	\end{align*}
\end{proof}
With this lemma, we can show that compositional products for continuous Weihrauch degrees are already given by our current definition:
\begin{corollary}
	For any problems $f$ and $g$, the continuous Weihrauch degree
	\begin{equation*}
		\max{}_{\cont} \{f' \circ g' \mid f' \cont f \text{ and } g' \cont g\}
	\end{equation*}
	exists and is equal to the continuous Weihrauch degree associated with $f * g$.
\end{corollary}

\begin{proof}
	First, let $f'$ and $g'$ be problems with $f' \weih f$ and $g' \weih g$ such that $f' \circ g' \weihEq f * g$ holds. Clearly, we have $f' \cont f$ and $g' \cont g$ and, thus, $f * g$ inhabits the set of problems that we are taking the maximum of. Now, we only have to show that any other problem in this set continuously Weihrauch reduces to $f * g$:
	Let $f'$ and $g'$ be problems with $f' \cont f$ and $g' \cont g$. By Lemma \ref{lem:combine_cont}, we have $f' \circ g' \weih f' * g' \cont f * g$. 
\end{proof}

The final ingredients for the proof of our Theorem are that, even in the context of continuous Weihrauch degrees, both $\lpo * \lpo$ and $\llpo * \llpo$ are strictly stronger than $\lpo$ and $\llpo$, respectively.
\begin{lemma}\label{lem:lpo2_stronger_than_lpo}\mbox{}
	\begin{enumerate}[label=(\roman*)]
		\item $\lpo * \lpo$ is not continuously Weihrauch reducible to $\lpo$.
		\item $\llpo * \llpo$ is not continuously Weihrauch reducible to $\llpo$.
	\end{enumerate}
\end{lemma}
These essentially follow from \cite[Theorem~3.8 and Theorem~5.4.2]{Weihrauch92}. In order to stay self-contained, let us quickly prove them ourselves. They are almost immediate from the following slightly stronger result:
\begin{lemma}\label{lem:llpo_times_llpo_not_red_to_lpo}
	$\llpo \times \llpo$ is not continuously Weihrauch reducible to $\lpo$.
\end{lemma}

\begin{proof}
	Let $h, k: \subseteq \nn \to \nn$ be partial continuous functions that realize the reduction $\llpo \times \llpo \cont \lpo$. Moreover, let $(y_n)_{n \in \n} \subseteq \nn$ be a list of sequences where $y_n$ only consists of zeros except for the value at index $n$, for any $n \in \n$. First, we show that there is an $n \in \n$ such that $x_n := k(\langle y_n, \zero \rangle)$ has a zero. Assume, for contradiction, that this is not the case. If the sequence $k(\langle \zero, \zero \rangle)$ has a zero, then it is clear by continuity of $k$ that such an $n$ exists. Otherwise, $\lpo$ must reply with the solution $\one$ for this instance, and $h$ produces $b_1, b_2 \in \{\zero, \one\}$ with $\langle b_1, b_2 \rangle = h(\langle \langle \zero, \zero \rangle, \one \rangle)$. W.l.o.g., assume that $b_1$ is equal to $\zero$. Then, by continuity of $h$ and the assumption that $h$ always produces valid solutions for $\llpo \times \llpo$, there must exist some $n \in \n$ such that $\langle b_1, b_2 \rangle = h(\langle \langle y_{2n}, \zero \rangle, \one \rangle)$ holds. However, $b_1 = \zero$ is not a valid solution for the instance $y_{2n}$ of $\llpo$. Thus, $\one$ must not be a valid solution for the instance $x_{2n} = k(\langle y_{2n}, \zero \rangle)$ of $\lpo$. We conclude that $x_{2n}$ must have a zero. The argument for $b_1 = \one$ works analogously by simply using $y_{2n+1}$ instead of $y_{2n}$.
	
	Let $n \in \n$ be such that $x_n$ has a zero. Now, we consider $(z_m)_{m \in \n} \subseteq \nn$ with $z_m := k(\langle y_n, y_m \rangle)$ for any $m \in \n$. Similar to before, let $b_1, b_2 \in \{\zero, \one\}$ be such that $\langle b_1, b_2 \rangle = h(\langle \langle y_n, 0 \rangle, \zero \rangle)$ holds. Notice that the solution for $x_n$ of $\lpo$ must be $\zero$ since $x_n$ contains a zero. W.l.o.g., assume that $b_2$ is equal to $\zero$. Now, by continuity of both $k$ and $h$, there must be some $m \in \n$ such that both $z_{2m}$ has a zero (because of $\lim_{m \to \infty} z_m = x_n$) and $\langle b_1, b_2 \rangle = h(\langle \langle y_n, y_{2m} \rangle, \zero \rangle)$ holds. Since $z_{2m}$ has a zero, $\zero$ is still the valid solution for the instance $z_{2m}$ of $\lpo$. But now, the reducibility tells us that $b_2 = \zero$ is a correct solution for the instance $y_{2m}$ of $\llpo$. This leads to a contradiction. The argument for $b_2 = \one$ works analogously by using $z_{2m+1}$ instead of $z_{2m}$.
\end{proof}
Now, we can prove the previous lemma:
\begin{proof}[Proof of Lemma \ref{lem:lpo2_stronger_than_lpo}]
	It is a classical result that $\llpo \weih \lpo$ holds (cf.~\cite[Theorem~4.2]{Weihrauch92}), which we can quickly verify: Let $k: \nn \to \nn$ be some computable function that maps any instance $x \in \nn$ of $\llpo$ to $k(x) \in \nn$ with $k(x)_n = 0$ if and only if $x_n \neq 0$. If $\lpo$ tells us that $k(x)$ has a zero, then $x$ has a non-zero member. Thus, we can simply search for it and determine whether all even positions (or odd positions) in $x$ are zeros. Otherwise, if $\lpo$ tells us that $k(x)$ does not have a zero, then $x = \zero$ must hold and any of $\zero$ or $\one$ is a valid solution for $\llpo$ of $x$.
	
	Under any of the assumptions $\lpo * \lpo \cont \lpo$ or $\llpo * \llpo \cont \llpo$, we can use the previous paragraph in order to derive $\llpo * \llpo \cont \lpo$. Thus, we have (cf.~\cite[Lemma 4.3]{BGM12})
	\begin{equation*}
		\llpo \times \llpo \weih (\id \times \llpo) \circ (\llpo \times \id) \weih \llpo * \llpo \cont \lpo\period
	\end{equation*}
	Now, we simply invoke Lemma \ref{lem:llpo_times_llpo_not_red_to_lpo} and derive a contradiction.
\end{proof}
With the previous results and the obvious reductions $\lpo \weih \lpo * \lpo$ and $\llpo \weih \llpo * \llpo$, Theorem \ref{thm:lpo_no_root} actually follows from this slightly more general result:
\begin{theorem}\label{thm:no_mth_root}
	Given a number $n \geq 2$, any problem $f$ with
	\begin{enumerate}[label=(\roman*)]
		\item $\lpo \cont f \contStrict \lpo^{[n]}$ or
		\item $\llpo \cont f \contStrict \llpo^{[n]}$
	\end{enumerate}
	does not have $m$-th roots for $m \geq n$.
\end{theorem}
In order to derive Theorem \ref{thm:lpo_no_root} from this, use $n := 2$ together with $f := \lpo$ or $f := \llpo$.
\begin{proof}
	Assume that $\lpo \weih f \contStrict \lpo^{[n]}$ holds and that $f$ has an $m$-th root~$r$. Now, assume for contradiction that $r$ is $1$-weakly continuous. Inductively, we can show that $r^{[i]}$ is $1$-weakly continuous for any $i \geq 1$. For the induction step, let $g \weih r$ and $h \weih r^{[i]}$ be such that $g \circ h$ has degree $r^{[i+1]}$. Since $r$ and $r^{[i]}$ are $1$-weakly continuous, Lemma \ref{lem:closure} transfers this to $g$ and $h$ and, hence, to $g \circ h$ and $r^{[i+1]}$. By the same lemma, we know that $f$ must be $1$-weakly discontinuous since $\lpo$ has this property by Lemma \ref{lem:discontinuous}. Finally, for $i := m$, this leads to a contradiction.
	
	Now, by Lemma \ref{lem:program_lpo}, we have $\lpo \cont r$ and, therefore, Lemma \ref{lem:combine_cont} yields the chain of reductions $\lpo^{[m]} \cont r^{[m]} \weihEq f$. Now, by assumption, this entails $\lpo^{[m]} \cont f \contStrict \lpo^{[n]}$. Since $m \geq n$ holds, this is a contradiction.
	
	The argument for $\llpo$ is the same, we simply have to replace ``$1$-weakly continuous'' and ``Lemma \ref{lem:program_lpo}'' by ``$2$-weakly continuous'' and ``Lemma \ref{lem:program_llpo}'', respectively.
\end{proof}
\begin{corollary}
	Both $\lpo$ and $\llpo$ are $*$-irreducible (cf.~\cite[Section~6.2]{BP18}) with respect to continuous Weihrauch degrees:
	Let $f$ and $g$ be problems such that the equality $\lpo \contEq f * g$ holds. Then, we have $\lpo \contEq f$ or $\lpo \contEq g$. The statement also holds true if we replace $\lpo$ with $\llpo$.
\end{corollary}
\begin{proof}
	We give the argument for $\lpo$.
	From $\lpo \contEq f * g$, we conclude that both $f$ and $g$ must have a non-empty domain. Thus, we have $1 \cont f$ and $1 \cont g$.
	Using the same argument as in Theorem \ref{thm:no_mth_root}, we have $\lpo \cont f$ or $\lpo \cont g$. In the former case, we derive $\lpo \cont f \cont f * 1 \cont f * g \cont \lpo$. In the latter case, we have $\lpo \cont g \cont 1 * g \cont f * g \cont \lpo$. The proof for $\llpo$ works analogously.
\end{proof}
The proof of Theorem \ref{thm:no_mth_root} can be turned into a second proof of Theorem \ref{thm:uncount_no_root} by considering problems $\lpo \times \dg_a$ for non-zero Turing degrees $a$, where $\dg_a$ is the constant multi-valued function that maps everything to all number sequences of degree $a$. Still, the proof from the previous section is much simpler and the considered problems $w_a$ have another interesting property: While they do not have roots in the Weihrauch lattice, they do have roots in the continuous Weihrauch lattice: For this, we prove $w_a \contEq \id$. The direction $\id \cont w_a$ is trivial and the reduction $w_a \cont \id$ is realized by continuous functions $h$ and $k$ with $h := c_x$ for a number sequence $x$ of Turing degree $a$ and $k := \id$. Clearly, $\id$ is its own $n$-th root for $n \geq 2$. Thus, $w_a$ has all roots in the continuous Weihrauch lattice, in contrast to $\lpo \times \dg_a$.

\section{Finitely many compositions of $\lpo$}

In this section, we want to prove Theorem \ref{thm:some_roots}. An important ingredient for this result is the following theorem:
\begin{theorem}\label{thm:lpo_n_strict_lpo_n_1}
	For any $n \in \n$, we have $\lpo^{[n]} \weihStrict \lpo^{[n+1]}$. The same result holds if we replace $\weihStrict$ with $\contStrict$.
\end{theorem}
To the best of our knowledge, a proof for this result does not appear in literature, yet. For our argument, we will need many ideas from \cite[Section~3]{Weihrauch92}. But first, we introduce the realizer-based definition of Weihrauch reducibility:
\begin{definition}
	Given a problem $f$, we call $r:\subseteq \nn \to \nn$ a \emph{realizer} of $f$ if and only if $r(x) \in f(x)$ for any $x \in \dom(f)$.
\end{definition}
For any instance $x \in \dom(f)$, a realizer chooses some solution in $f(x)$ explicitly.
\begin{lemma}
	Given two problems $f$ and $g$, we have $f \weih g$ if and only if there are computable functions $h, k:\subseteq \nn \to \nn$ such that any realizer of $h \circ \langle \id, r \circ k \rangle$ is also a realizer of $f$. Similarly, we have $f \cont g$ if and only if the statement holds for continuous functions $h$ and $k$.
\end{lemma}
The proof is based on a simple application of the axiom of choice (cf.~\cite[Proposition~11.3.2]{BGP}). Next, we define $k$-continuous problems (cf.~\cite[Definition~3.3]{Weihrauch92}):
\begin{definition}
	Given $k \in \n$, we say that a problem $f$ is $k$-continuous if and only if it has a realizer $r$ such that we can find a partition $P$ of $\dom(r)$ with $|P| \leq k$ such that $r$ is continuous when restricted to any element of $P$.
\end{definition}
\begin{lemma}\label{lem:lpo_2_continuous}
	$\lpo$ is $2$-continuous.
\end{lemma}
\begin{proof}
	Let $X_1 \subset \nn$ be the set of all number sequences that do not have a zero. Let $X_2 \subset \nn$ be the set of all number sequences that do have a zero. Clearly, $P := \{X_1, X_2\}$ is a partition of $\nn$ and $\lpo$ is constant on each restriction $X_1$ and~$X_2$.
\end{proof}
\begin{lemma}\label{lem:k_continuous_transports}
	Let $f$ and $g$ be problems.
	\begin{enumerate}[label=(\roman*)]
		\item If $f \cont g$ and $g$ is $k$-continuous for some $k \in \n$, then $f$ is also $k$-continuous.
		\item If $f$ and $g$ are $k$-continuous for some $k \in \n$, then there is some $l \in \n$ such that $f \circ g$ is $l$-continuous.
	\end{enumerate}
\end{lemma}
\begin{proof}
	For (i), assume that $f$ and $g$ are $i$-continuous for some $i \in \n$. Let $h, k$ realize the Weihrauch reduction $f \cont g$. Let $r$ be a realizer of $g$ that is continuous on every element of the partition $P$ of $\dom(g)$ with $|P| \leq k$. Clearly, this property is transferred to the realizer $h \circ \langle \id, r \circ k \rangle$ of $f$ with partition $P' := \{k^{-1}(X) \mid X \in P\}$.
	
	For (ii), assume that $f$ and $g$ are $k$-continuous for some $k \in \n$. Let $r_f$ and $r_g$ be realizers of $f$ and $g$ such that $r_f$ and $r_g$ are continuous on every element of the partitions $P_f$ and $P_g$ with $|P_f|, |P_g| \leq k$, respectively. Define
	\begin{equation*}
		P'_g := \{X \cap g^{-1}(Y) \mid X \in P_g \text{ and } Y \in P_f\} \setminus \{\emptyset\}\period
	\end{equation*}
	Clearly, $r_g$ is still continuous on every element of $P'_g$. Moreover, $r_g(X)$ is a subset of some element in $P_f$ for any $X \in P'_g$. Thus, $r_f \circ r_g$ is continuous on every $X \in P'_g$. Since $r_f \circ r_g$ is a realizer of $f \circ g$, we conclude that $f \circ g$ is $|P'_g|$-continuous.
\end{proof}
With this, we know that if a problem is $k$-continuous for $k \in \n$, then this property is shared with all other problems of the same degree.
\begin{corollary}\label{cor:lpo_n_finitely_continuous}
	For any $n \in \n$, there is some number $k \in \n$ such that $\lpo^{[n]}$ is $k$-continuous.
\end{corollary}
\begin{proof}
	For $n := 0$, we have $\lpo^{[n]} \weihEq \id$, which is clearly ($1$-)continuous. Now, for the induction step, let $f \weih \lpo$ and $g \weih \lpo^{[n]}$ such that $f \circ g \weihEq \lpo^{[n+1]}$ holds. By Lemmas \ref{lem:lpo_2_continuous} and \ref{lem:k_continuous_transports}, we know that $f \circ g$ must be $k$-continuous for some $k \in \n$.
\end{proof}
\begin{lemma}\label{lem:lpo_n_not_n_continuous}
	Given $n \geq 1$, let $\lpo^{n}$ denote the combination of $n$-many copies of $\lpo$ using the $\times$-operator. This problem is not $n$-continuous.
\end{lemma}
\begin{proof}
	Let $f:\subseteq \nn \to \nn$ be the problem whose domain contains exactly those number sequences with at most $n$-many zeros. Given such a number sequence $x \in \dom(f)$, we define $f(x) := \mathbf{m}$ for the number $m$ of zeros in $x$.
	
	First, we show that $f$ reduces to $\lpo^{n}$. For every $i < n$, let $k_i: \nn \to \nn$ be the computable function that maps any number sequence $x \in \nn$ to the sequence $k_i(x)$ that is defined like $x$ but where the first $i$-many zeros are skipped. We apply $\lpo^{n}$ to $\langle k_0(x), \dots, k_{n-1}(x) \rangle$ to some arbitrary $x \in \dom(f)$. If $\lpo^{n}$ tells us that $k_0(x)$, i.e.~$x$ itself, does not have a zero, then we know that $f(x) = \zero$ holds. Otherwise, let $i < n$ be the largest index such that $k_i(x)$ has no zero. Thus, skipping $i$-many zeros in $x$ results in a sequence without zeros but skipping $(i+1)$-many zeros in $x$ yields a sequence with zeros (or, of course, we have $i = n$). We conclude that $x$ has exactly $i$-many zeros. Therefore, we have $f(x) = \mathbf{i}$.
	
	Now, we apply \cite[Theorem~3.5]{Weihrauch92}, which tells us that $f$ is not $n$-continuous. Finally, with $f \weih \lpo^{n}$ and Lemma \ref{lem:k_continuous_transports}, we conclude that $\lpo^{n}$ is not $n$-continuous.
\end{proof}
\begin{proof}[Proof of Theorem \ref{thm:lpo_n_strict_lpo_n_1}]
	Assume for contradiction that there is some $n \in \n$ such that $\lpo^{[n]} \weihStrict \lpo^{[n+1]}$ or $\lpo^{[n]} \contStrict \lpo^{[n+1]}$ does not hold. In this case, we have $\lpo^{[n]} \contEq \lpo^{[n+1]}$. Moreover, by composing with $m$-many copies of $\lpo$, we get $\lpo^{[n+m]} \contEq \lpo^{[n+m+1]}$. Thus, we have $\lpo^{[n]} \contEq \lpo^{[m]}$ for any $m > n$.
	
	By Corollary \ref{cor:lpo_n_finitely_continuous}, we find some number $k \in \n$ such that $\lpo^{[n]}$ is $k$-continuous. By Lemma \ref{lem:lpo_n_not_n_continuous}, we know that $\lpo^{k}$ is not $k$-continuous. This problem reduces to $\lpo^{[k]}$ (cf.~\cite[Lemma 4.3]{BGM12}). Let $m := \max(k, n+1)$. With $\lpo^{[k]} \weih \lpo^{[m]}$, we conclude that $\lpo^{[m]}$ is not $k$-continuous. This leads to a contradiction: Since $m > n$ holds, $\lpo^{[n]}$ and $\lpo^{[m]}$ must be equivalent degrees.
\end{proof}
\begin{proof}[Proof of Theorem \ref{thm:some_roots}]
	Let $m \geq 2$. We apply Theorem \ref{thm:no_mth_root} for $f := \lpo^{[m]}$ and $n := m + 1$. By Theorem \ref{thm:lpo_n_strict_lpo_n_1}, we know that this a valid instance since the reductions $\lpo \cont \lpo^{[m]} \contStrict \lpo^{[m+1]}$ hold. Now, we conclude that $\lpo^{[m]}$ does not have an $(m+1)$-th root. However, it clearly has an $m$-th root.
\end{proof}

\printbibliography

\end{document}